\documentclass{amsart}
\pdfoutput=1
\usepackage{float}
\restylefloat{figure}
\usepackage{mathtools}
\usepackage{amssymb}
\usepackage{mathrsfs}
\usepackage{bm}
\usepackage[all]{xy}
\UseComputerModernTips
\CompileMatrices
\usepackage[colorlinks,allcolors=blue]{hyperref}
\usepackage{ifthen}
\usepackage{latexsym}
\usepackage{tikz}
\usepackage{tikz-3dplot}
\usetikzlibrary{backgrounds}
\allowdisplaybreaks
\numberwithin{equation}{section}
\newtheorem{theorem}[equation]{Theorem}
\newtheorem{lemma}[equation]{Lemma}
\newtheorem{proposition}[equation]{Proposition}
\newtheorem{corollary}[equation]{Corollary}
\newtheorem*{bezoutOne}{B\'ezout's Theorem, Part I}
\newtheorem*{bezoutTwo}{B\'ezout's Theorem, Part II}
\theoremstyle{definition}
\newtheorem{definition}[equation]{Definition}

\newtheorem{context}[equation]{B\'ezout Context}
\theoremstyle{remark}
\newtheorem{remark}[equation]{Remark}

\renewcommand{\phi}{\varphi}
\DeclareMathSymbol{\boxprod}{\mathbin}{AMSa}{"03} 
\DeclareMathSymbol{\mixprod}{\mathbin}{AMSa}{"4F} 

\newcommand{\dirsum}{\oplus}

\newcommand{\disjunion}{\sqcup}

\newcommand{\dual}{^\vee}

\newcommand{\includesin}{\hookrightarrow}

\newcommand{\iso}{\cong}
\newcommand{\Mackey}[1]{{\underline {#1}}}

\newcommand{\smsh}{\wedge}

\newcommand{\tensor}{\otimes}
\newcommand{\union}{\cup}

\newcommand{\C}{{\mathbb C}}

\newcommand{\PP}{\mathbb{P}}
\newcommand{\R}{{\mathbb R}}

\newcommand{\Z}{\mathbb{Z}}
\newcommand{\ZZ}{\mathbb{Z}}
\newcommand{\GG}{{C_2}}
\newcommand{\HA}{H_\GG^{RO(\GG)}}
\newcommand{\cwt}{\zeta_1}
\newcommand{\cxwt}{\zeta_0}
\newcommand{\cwd}{\,\widehat{c}_\omega}
\newcommand{\cxwd}{\,\widehat{c}_{\chiw}}
\newcommand{\cd}[1][{}]{\,\widehat{c}^{\;#1}}
\newcommand{\Cpq}[2]{\C^{#1+#2\sigma}}
\newcommand{\Cq}[1]{\C^{#1\sigma}}
\newcommand{\Xpq}[2]{\PP(\Cpq{#1}{#2})}
\newcommand{\Xp}[1]{\PP(\C^{#1})}
\newcommand{\Xq}[1]{\PP(\Cq{#1})}

\newcommand{\chiw}{\chi\omega}

\newcommand{\I}{{\mathrm{I}}}
\newcommand{\II}{{\mathrm{II}}}
\newcommand{\III}{{\mathrm{III}}}
\newcommand{\IV}{{\mathrm{IV}}}

\DeclareMathOperator{\grad}{grad}
\newcommand{\Borel}{BH}
\newcommand{\copt}{{\mathbb H}}
\providecommand*{\donothing}[1]{}

\begin{document}

\title{An Algebraic $\GG$-equivariant B\'{e}zout's theorem}
\author{Steven R. Costenoble}
\address{Steven R. Costenoble, Department of Mathematics\\Hofstra University\\
   Hempstead, NY 11549}
\email{Steven.R.Costenoble@Hofstra.edu}
\author{Thomas Hudson}
\address{Thomas Hudson, College of Transdisciplinary Studies, DGIST, 
Daegu, 42988, Republic of Korea}
\email{hudson@digst.ac.kr}
\author{Sean Tilson}
\address{Sean Tilson}
\email{Sean.Tilson@gmail.com}

\begin{abstract}
B\'ezout's theorem, nonequivariantly, can be interpreted as a calculation of the Euler class
of a sum of line bundles over complex projective space,
expressing it in terms of the rank of the bundle and its degree.
We give here a generalization to the $C_2$-equivariant context, 
using the calculation of the cohomology of a $C_2$-complex projective space
from an earlier paper.
We use ordinary $\GG$-cohomology with Burnside ring coefficients and an extended grading
necessary to define the Euler class, which
we express in terms of the equivariant rank of the bundle
and the degrees of the bundle and its fixed subbundles.
We do similar calculations using constant $\Z$ coefficients and
Borel cohomology and compare the results.
\end{abstract}

\makeatletter
\@namedef{subjclassname@2020}{%
  \textup{2020} Mathematics Subject Classification}
\makeatother

\maketitle

\tableofcontents
\section*{Introduction}

Suppose that we have
$n$ nonzero homogeneous polynomials $f_i$, $1\leq i \leq n$, in $N$ variables, where $n < N$,
and let $d_i$ be the degree of $f_i$.
If $\PP^{N-1}$ is the complex projective space,
we can consider each $f_i$ as giving a section of the complex line bundle $O(d_i)$,
the $d_i$-fold tensor power of the dual of the tautological line bundle over $\PP^{N-1}$.
Each $f_i$ determines a hypersurface $H_i \subset \PP^{N-1}$, its zero locus.
In this context, the (nonequivariant) B\'ezout theorem, as given in \cite{IntersectionFulton},
for example, can be stated in several ways.
One is a purely algebraic statement: In the cohomology ring 
\[
    H^*(\PP^{N-1};\Z) \iso \Z[ \,\widehat c\,]/\langle \,\widehat c^{\;N}\rangle,
\]
we have that the Euler class of $F = O(d_1)\dirsum O(d_2)\dirsum\cdots\dirsum O(d_n)$ is
\[
    e(F) = \Delta\,\widehat c^{\;n}\,,
\]
where $\Delta = d_1 d_2 \cdots d_n$ and $\widehat c = e(O(1))$.
As a consequence, $e(F)$ determines and is completely determined by the rank $n$ of $F$
(that is, the complex dimension of each of its fibers)
and its degree $\Delta$.

A similar statement is true with the Chow ring in place of ordinary cohomology
(they are isomorphic in this case),
which gives us the geometric aspect of B\'ezout's theorem:
Generically, the intersection of the hypersurfaces $H_i$, counted with multiplicities, is
rationally equivalent to $\Delta$ copies of $\PP^{N-n-1}$.
In the classical case, when $n = N-1$, the hypersurfaces intersect in $\Delta$ points.

In \cite{CHTFiniteProjSpace}, we began to examine how this generalizes in the presence
of an action of the two-element group $\GG$.
Let $\C$ denote the trivial complex representation of $\GG$ and let
$\C^\sigma$ denote the nontrivial representation.
If $p\geq 0$ and $q\geq 0$ are integers, 
let $\Cpq pq$ be the sum of $p$ copies of $\C$ and $q$ copies of $\C^\sigma$, and let
$\Xpq pq$ be its (complex) projective space, a $\GG$-space.
Using the equivariant ordinary cohomology with extended grading defined in \cite{CostenobleWanerBook},
we computed the cohomology of $\Xpq pq$ in \cite{CHTFiniteProjSpace}
with Burnside ring coefficients.
We also gave the zero-dimensional version of an equivariant B\'ezout theorem,
showing that the equivariant Euler class in equivariant ordinary cohomology allows us to
determine the finite $\GG$-set in $\Xpq pq$ given by the intersection of $p+q-1$
equivariant hypersurfaces.

Our goal in this paper is to generalize the algebraic part of this $\GG$-B\'ezout theorem
to higher dimensions.
In a followup paper we will discuss the geometric meaning of the Euler class that we calculate here.

Let us set up the context more precisely.
If $F$ is a nonequivariant vector bundle over $\PP^N$, its Euler class has the form
$e(F) = \Delta \cd[n]$, where $n$ is the rank of $F$ and we set $\Delta = 0$ if $n \geq N$.
We call $\Delta$ the {\em degree} of $F$.
By the nonequivariant B\'ezout theorem, if $F$ is a sum of line bundles of the form $O(d)$ and $n < N$,
then $\Delta$ is the product of the degrees $d$.

Now suppose that we have $n < p+q$ $\GG$-line bundles over $\Xpq pq$ with direct sum $F$.
We let $\Delta$ be the nonequivariant degree of $F$.
We can also consider the fixed-set bundle $F^\GG$ over
$\Xpq pq^\GG = \Xp p \disjunion \Xq q$.
Let $n_0$ denote the rank of the restriction of $F^\GG$ to $\Xp p$ and let $\Delta_0$ be its degree.
We know that $n_0 \leq n$, and,
to keep the situation geometrically meaningful,
we would like the generic intersection of the corresponding hypersurfaces in $\Xp p$ to have dimension
no more than the dimension of the intersection of all the hypersurfaces in $\Xpq pq$.
For that, we require that $p-n_0-1 \leq p+q-n-1$, that is,
$n_0 \geq n - q$.
Similarly, let $n_1$ denote the rank of $F^\GG$ over $\Xq q$ and let $\Delta_1$ be its degree;
we require that $n_1 \geq n - p$. We record these notations and conditions for later reference.

\begin{context}\label{context}
$F$ is the sum of $n$ $\GG$-line bundles over $\Xpq pq$
and $\Delta$ is its nonequivariant degree.
The restriction of $F^\GG$ to $\Xp p$ has rank $n_0$ and degree $\Delta_0$,
while its restriction to $\Xq q$ has rank $n_1$ and degree $\Delta_1$.
We assume that
\begin{align*}
    n &< p + q \\
    n - q &\leq n_0 \leq n \\
    \mathllap{\text{and}\qquad} n - p &\leq n_1 \leq n.
\end{align*}
We call the triple $(\Delta,\Delta_0,\Delta_1)$ the
{\em $\GG$-degrees} of $F$.
\end{context}



\begin{bezoutOne}
In the context above, the Euler class $e(F)$ is completely determined by
the ranks $(n,n_0,n_1)$ and the degrees $(\Delta,\Delta_0,\Delta_1)$.
Moreover, these ranks and degrees can be recovered from $e(F)$.
The ranks are additive and the degrees are multiplicative.
\end{bezoutOne}

This will be proved as Theorem~\ref{thm:bezoutone}.
When we say that the degrees are multiplicative, we really mean the following:
Suppose that we have two such bundles $F$ and $F'$, with ranks $(n,n_0,n_1)$ and $(n',n'_0,n'_1)$, respectively,
and corresponding degrees. We assume that $F\dirsum F'$ still satisfies the conditions of the B\'ezout context
above. This allows the possibility that $n_0 + n'_0 \geq p$, in which case the
corresponding degree is not $\Delta_0\Delta'_0$ but 0, and similarly if $n_1 + n'_1 \geq q$.

The nonequivariant B\'ezout theorem can also be viewed as expressing $e(F)$ in terms of the basis
of the cohomology of $\PP^{N-1}$ given by the powers of $\cd$. In any given grading, there is at most
one such power, so there is only one coefficient to specify, which turns out to be the degree $\Delta$.
Equivariantly, the result is more complicated.
The cohomology of $\Xpq pq$ is free over the $RO(\GG)$-graded equivariant cohomology of a point,
but the cohomology of a point is no longer concentrated in grading 0.
As shown in \cite{CHTFiniteProjSpace},
in any given grading of the cohomology of $\Xpq pq$, there are up to $p+q$ basis elements
that can contribute, so an element potentially requires a $(p+q)$-tuple of coefficients to specify.
Our second main result is summarized as follow.

\begin{bezoutTwo}
In the context above, the Euler class $e(F)$ is the linear combination of at most three basis elements.
\end{bezoutTwo}

This is proved as Theorem~\ref{thm:bezouttwo}, which also gives the details as to which three
basis elements are involved and what their coefficients are.
The three basis elements are determined by ($p$ and $q$ and) the ranks $(n,n_0,n_1)$.
The coefficients are determined by the degrees $(\Delta,\Delta_0,\Delta_1)$, but
are not simply equal to them.

This paper is structured as follows.
In Section~\ref{sec:cohomology}, we review the cohomology of $\Xpq pq$ as computed in \cite{CHTFiniteProjSpace},
including our preferred basis.
In Section~\ref{sec:ideal} we give the main results, proving the two theorems above.
There are two other equivariant ordinary cohomology theories in common use,
cohomology with constant $\Z$ coefficients and Borel cohomology.
In Section~\ref{sec:Z} we discuss how the computation changes if we use constant $\Z$ coefficients rather than Burnside ring coefficients,
and in Section~\ref{sec:Borel} we discuss the similar computation in Borel cohomology.
There are maps from cohomology with Burnside ring coefficients to cohomology
with constant $\ZZ$ coefficients, and from that theory to Borel cohomology, both respecting Euler classes, 
and we will see that
the Euler classes in the latter two theories carry less information than
the Euler class in cohomology with Burnside ring coefficients.
In particular, we cannot recover the degrees $\Delta_0$ and $\Delta_1$ from
the Euler class in cohomology with constant $\Z$ coefficients or the class in Borel cohomology.

\vspace{0.2 cm}

 \textit{Acknowledgements:} The second and third author were partially supported by the DFG through the SPP 1786: \textit{Homotopy theory and Algebraic Geometry}, Project number 405468058: \textit{$C_2$-equivariant Schubert calculus of homogeneous spaces}.
The research was conducted in the framework of the research training group
\emph{GRK 2240: Algebro-Geometric Methods in Algebra, Arithmetic and Topology},
which is funded by the DFG, while the second and third authors were affiliated to the Bergische Universit\"{a}t Wuppertal.

\section{The cohomology of $\Xpq pq$}\label{sec:cohomology}

\subsection{Ordinary cohomology}
We will use $\GG$-equivariant ordinary cohomology with the extended grading developed in \cite{CostenobleWanerBook}.
This is an extension of Bredon's ordinary cohomology to be graded on representations of the fundamental groupoids
of $\GG$-spaces.
We review here some of the notation and computations we will be using.
A more detailed description of this theory can be found in \cite{CHTFiniteProjSpace}.

For an ex-$\GG$-space $Y$ over $X$, we write $H_\GG^{RO(\Pi X)}(Y;\Mackey T)$ for the
ordinary cohomology of $Y$ with coefficients in a Mackey functor $\Mackey T$, graded
on $RO(\Pi X)$, the representation ring of the fundamental groupoid of $X$.
Through most of this paper we will use the Burnside ring Mackey functor $\Mackey A$ as the coefficients,
and write simply $H_\GG^{RO(\Pi X)}(Y)$.

In \cite{CostenobleWanerBook} and \cite{CHTFiniteProjSpace} we considered cohomology to be Mackey functor--valued,
which is useful for many computations, and wrote $\Mackey H_\GG^{RO(\Pi X)}(Y)$ for the resulting theory.
In this paper we will be concentrating on the values at level $\GG/\GG$, and write
$H_\GG^{RO(\Pi X)}(Y) = \Mackey H_\GG^{RO(\Pi X)}(Y)(\GG/\GG)$. However, we will still refer to the
restriction functor $\rho$ from equivariant cohomology to nonequivariant cohomology, and the transfer
map $\tau$ going in the other direction.

For all $X$ and $Y$, $H_\GG^{RO(\Pi X)}(Y)$ is a graded module over 
\[
    \copt = \copt^{RO(\GG)} = H_\GG^{RO(\GG)}(S^0),
\]
the cohomology of a point.
The grading on the latter is just $RO(\GG)$, the real representation ring of $\GG$, which is free abelian
on $1$, the class of the trivial representation $\R$, and $\sigma$, the class of the
sign representation $\R^\sigma$. The cohomology of a point was calculated by Stong in an unpublished
manuscript and first published by Lewis in \cite{LewisCP}. We can picture the calculation as 
in Figure~\ref{fig:cohompt},
in which a group in grading $a+b\sigma$ is plotted at the point $(a,b)$, and the spacing of the grid lines
is 2 (which is more conventient for other graphs we will give).
The box at the origin is a copy of $A(\GG)$, the Burnside ring of $\GG$,
closed circles are copies of $\Z$, and open circles are copies of $\Z/2$.
\begin{figure}
\begin{tikzpicture}[x=4mm, y=4mm]
	\draw[step=2, gray, very thin] (-7.8, -7.8) grid (7.8, 7.8);
	\draw[thick] (-8, 0) -- (8, 0);
	\draw[thick] (0, -8) -- (0, 8);
    \node[right] at (8,0) {$a$};
    \node[above] at (0,8) {$b\sigma$};

    \fill (-0.3, -0.3) rectangle (0.3, 0.3);
    \fill (0, -7) circle(0.2);
    \fill (0, -6) circle(0.2);
    \fill (0, -5) circle(0.2);
    \fill (0, -4) circle(0.2);
    \fill (0, -3) circle(0.2);
    \fill (0, -2) circle(0.2);
    \fill (0, -1) circle(0.2);
    \fill (0, 1) circle(0.2);
    \fill (0, 2) circle(0.2);
    \fill (0, 3) circle(0.2);
    \fill (0, 4) circle(0.2);
    \fill (0, 5) circle(0.2);
    \fill (0, 6) circle(0.2);
    \fill (0, 7) circle(0.2);

    \fill (-2, 2) circle(0.2);
    \fill (-4, 4) circle(0.2);
    \fill (-6, 6) circle(0.2);

    \draw[fill=white] (-2, 3) circle(0.2);
    \draw[fill=white] (-2, 4) circle(0.2);
    \draw[fill=white] (-2, 5) circle(0.2);
    \draw[fill=white] (-2, 6) circle(0.2);
    \draw[fill=white] (-2, 7) circle(0.2);
    \draw[fill=white] (-4, 5) circle(0.2);
    \draw[fill=white] (-4, 6) circle(0.2);
    \draw[fill=white] (-4, 7) circle(0.2);
    \draw[fill=white] (-6, 7) circle(0.2);

    \fill (2, -2) circle (0.2);
    \fill (4, -4) circle (0.2);
    \fill (6, -6) circle (0.2);
    \draw[fill=white] (3, -3) circle(0.2);
    \draw[fill=white] (5, -5) circle(0.2);
    \draw[fill=white] (7, -7) circle(0.2);
    
    \draw[fill=white] (3, -4) circle(0.2);
    \draw[fill=white] (3, -5) circle(0.2);
    \draw[fill=white] (3, -6) circle(0.2);
    \draw[fill=white] (3, -7) circle(0.2);
    \draw[fill=white] (5, -6) circle(0.2);
    \draw[fill=white] (5, -7) circle(0.2);

    \node[right] at (0,1) {$e$};
    \node[below left] at (-2,2) {$\xi$};
    \node[above right] at (2,-2) {$\tau(\iota^{-2})$};
    \node[left] at (0,-1) {$e^{-1}\kappa$};

\end{tikzpicture}
\caption{$\copt^{RO(\GG)}$}\label{fig:cohompt}
\end{figure}

Recall that $A(\GG)$ is the Grothendieck group of finite $\GG$-sets, with multiplication given by products of sets.
Additively, it is free abelian on the classes of the orbits of $\GG$, for which we will write
$1 = [\GG/\GG]$ and $g = [\GG/e]$. The multiplication is given by $g^2 = 2g$.
We will also write $\kappa = 2 - g$. Other important elements are shown in the figure:
The group in degree $\sigma$ is generated by an element $e$,
while the group in degree $-2 + 2\sigma$ is generated by an element $\xi$.
The groups in the second quadrant are generated by the products $e^m\xi^n$, with $2e\xi = 0$.
We have $g\xi = 2\xi$ and $ge = 0$.
The groups in gradings $-m\sigma$, $m\geq 1$, are generated by elements $e^{-m}\kappa$, so named
because $e^m\cdot e^{-m}\kappa = \kappa$. We also have $ge^{-m}\kappa = 0$.

To explain $\tau(\iota^{-2})$, we think for moment about the nonequivariant cohomology
of a point. If we grade it on $RO(\GG)$, we get
$H^{RO(\GG)}(S^0;\Z) \iso \Z[\iota^{\pm 1}]$, where $\deg \iota = -1 + \sigma$.
(Nonequivariantly, we cannot tell the difference between $\R$ and $\R^\sigma$.)
We have $\rho(\xi) = \iota^2$ and $\tau(\iota^2) = g\xi = 2\xi$.
Note also that $\tau(1) = g$.
In the fourth quadrant we have that the group in grading $n(1-\sigma)$, $n\geq 2$, is
generated by $\tau(\iota^{-n})$.
The remaining groups in the fourth quadrant will not concern us here.
For more details, see \cite{Co:BGU1preprint} or \cite{CHTFiniteProjSpace}.

\subsection{The cohomology of projective space}
As described in the introduction, the form of B\'ezout's theorem we shall give expresses
the Euler class of a bundle over $\Xpq pq$ in terms of a basis of its cohomology.
we now review the structure of that cohomology as calculated in \cite{CHTFiniteProjSpace}.

Write $B = \Xpq \infty\infty$.
Its fixed set is
\[
    B^\GG = \Xp p \disjunion \Xq q = B^0 \disjunion B^1,
\]
where we use the indices 0 and 1 to evoke the trivial and nontrivial representation
of $\GG$, respectively.
(We will use this convention throughout, that a subscript 0 refers to something related to
$B^0$ and subscript 1 refers to something related to $B^1$.)
Representations of $\Pi B$ are determined by their restrictions to $B^0$ and $B^1$,
which are elements of $RO(\GG)$ that must have the same nonequivariant rank and
the same parity for the ranks of their fixed point representations.
As a result, we can write
\[
    RO(\Pi B) = \Z\{1,\sigma,\Omega_0,\Omega_1\}/\langle \Omega_0 + \Omega_1 = 2\sigma - 2\rangle,
    \]
where $\Omega_0$ is the representation whose value on $B^0$ is $2\sigma - 2$  and
on $B^1$ is $0$;
while $\Omega_1$ is the representation whose value on $B^0$ is $0$ and
on $B^1$ is $2\sigma-2$.
For any $\alpha\in RO(\Pi B)$, write $|\alpha|\in\ZZ$ for its underlying nonequivariant rank,
and $\alpha_0$ and $\alpha_1 \in RO(\GG)$ for its retrictions to $B^0$ and $B^1$, respectively.
What we said above can be phrased as: $\alpha$ is completely determined by the triple of ranks
$(|\alpha|,|\alpha_0^\GG|,|\alpha_1^\GG|)$, where the latter two ranks have the same parity.

We think of the finite projective spaces as spaces over $B$ by the evident inclusions
$\Xpq pq \to \Xpq\infty\infty$, so will grade their cohomologies on $RO(\Pi B)$.
Let $\omega$ denote the tautological line bundle over $B$, let $\omega\dual$ be its dual bundle, let
$\chi\omega = \omega\tensor_\C \C^\sigma$, and let $\chi\omega\dual$ be the dual of $\chi\omega$.
We will also use the notation from algebraic geometry in which $\omega = O(-1)$ and $\omega\dual = O(1)$;
we write $\chi O(-1) = \chi\omega$ and $\chi O(1) = \chi\omega\dual$.

Associated to any bundle over $B$ is a representation in $RO(\Pi B)$ that we think of as the
equivariant rank of the bundle; this representation is given by the fiber representations
over $B^0$ and $B^1$. In the case of $\omega$ and $\chi\omega$, we have
\begin{align*}
    \omega &= 2 + \Omega_1 \\
    \chi\omega &= 2 + \Omega_0,
\end{align*}
where we write $\omega$ and $\chi\omega$ again for the associated elements of $RO(\Pi B)$.

Let $\cwd$ and $\cxwd$ denote the Euler classes of $\omega\dual$ and $\chi\omega\dual$, respectively.
The cohomology of $\Xpq\infty\infty$ was calculated in \cite{Co:BGU1preprint} as follows.

\begin{theorem}
$H_{\GG}^{RO(\Pi B)}(B_+)$ is an algebra over $\copt$ generated by the 
Euler classes $\cwd$ and $\cxwd$ together with classes $\cxwt$ and $\cwt$. 
These elements live in gradings
\begin{alignat*}{4}
 \grad\cwd &= \omega  \qquad&  \grad\cxwd &= \chiw  \\
 \grad\cwt &= \omega - 2  \qquad& \grad\cxwt &= \chiw-2 
\end{alignat*}
They satisfy the relations
\begin{align*}
		\cxwt \cwt &= \xi \\
        \mathllap{\text{and}\quad}\cwt \cxwd &= (1-\kappa)\cxwt \cwd + e^2,
\end{align*}
and these relations completely determine the algebra.
Moreover, $H_{\GG}^{RO(\Pi B)}(B_+)$ is free as a module over $\copt$.
\qed
\end{theorem}

Pulling back along the inclusion $\Xpq pq \includesin \Xpq\infty\infty$,
the cohomology of $\Xpq pq$ contains elements
we will also call $\cwd$, $\cxwd$, $\cxwt$, and $\cwt$.
In \cite{CHTFiniteProjSpace}, we showed the following.

\begin{theorem}[{\cite[Theorem A]{CHTFiniteProjSpace}}]\label{thm:cohomStructure}
Let $0 \leq p < \infty$ and $0 \leq q < \infty$ with $p+q > 0$.
Then $H_{\GG}^{RO(\Pi B)}(\Xpq{p}{q}_+)$ is a free module over $\copt$.
As a (graded) commutative algebra over $\copt$, the ring
$H_{\GG}^{RO(\Pi B)}(\Xpq{p}{q}_+)$ is generated by $\cwd$, $\cxwd$, $\cxwt$, and $\cwt$,
together with the following classes:
$\cwd^{\;p}$ is infinitely divisible by $\cxwt$, meaning that, for $k\geq 1$,
there are unique elements $\cxwt^{-k}\cwd^{\;p}$ such that
\[
 \cxwt^k \cdot \cxwt^{-k} \cwd^{\;p} = \cwd^{\;p}.
\]
Similarly, $\cxwd^{\;q}$ is infinitely divisible by $\cwt$, meaning that, for $k\geq 1$,
there are unique elements $\cwt^{-k}\cxwd^{\;q}$ such that
\[
 \cwt^k \cdot \cwt^{-k} \cxwd^{\;q} = \cxwd^{\;q}.
\]
The generators satisfy the following further relations:
\begin{align*}
	\cxwt \cwt &= \xi, \\
	\cwt \cxwd &= (1-\kappa)\cxwt \cwd + e^2 \\
    \mathllap{\text{and}\qquad} \cwd^{\;p} \cxwd^{\;q} &= 0.
\end{align*}
\qed
\end{theorem}

We also
gave an explicit basis for $H_\GG^{RO(\Pi B)}(\Xpq pq_+)$ over $\copt$,
which we can describe as follows. 
We define sets $F_{p,q}(m)$, recursively on $p$ and $q$, that give bases for $H_\GG^{m\omega+RO(\GG)}(\Xpq pq_+)$.
For $m\in\Z$, let
\begin{align*}
 F_{p,0}(m) &:= \{ \cwt^m, \cwt^{m-1}\cwd, \cwt^{m-2}\cwd^{\;2},\ldots, \cwt^{m-p+1}\cwd^{\;p-1} \} \\
\intertext{and}
 F_{0,q}(m) &:= \{ \cxwt^m, \cxwt^{m-1}\cxwd, \cxwt^{m-2}\cxwd^{\;2}, \ldots, \cxwt^{m-q+1}\cxwd^{\;q-1} \}.
\end{align*}
(Note that $\cwt$ is invertible in the first case and $\cxwt$ is invertible in the second.)
For $p, q > 0$ we then define
\[
 F_{p,q}(m) := 
  \begin{cases}
    \{ \cwt^m \} \union i_! F_{p-1,q}(m-1) & \text{if\ \,$m\geq 0$} \\
    \{ \cxwt^{|m|} \} \union j_! F_{p,q-1}(m+1) & \text{if\,\ $m < 0$,}
  \end{cases}
\]
where $i\colon \Xpq{p-1}q \to \Xpq pq$ and $j\colon \Xpq p{(q-1)} \to \Xpq pq$ are the inclusions.
The pushforward $i_!$ is given algebraically by multiplication by $\cwd$ and $j_!$ is multiplication by $\cxwd$.

It is possible from this description to write down the bases explicitly, but the results are messy, having to be broken
down by cases depending on where $m$ falls in relation to $p$ and $q$;
this is done in \cite[Proposition~4.7]{CHTFiniteProjSpace}.
However, we can make the following general statements.
\begin{enumerate}
    \item For fixed $p$, $q$, and $m$, there are exactly $p+q$ basis elements lying in $H_\GG^{m\omega+RO(\GG)}(\Xpq pq_+)$.
    \item Those basis elements have gradings of the form $m(\omega-2) + 2a_i + 2b_i\sigma$, $0\leq i \leq p+q-1$, where
            $a_i + b_i = i$.
    \item The basis element with grading $m(\omega-2) + 2a + 2b\sigma$ restricts to the nonequivariant class $\widehat c^{\;a+b}$, where
            $\widehat c$ is the first nonequivariant Chern class of $O(1)$.
    \item For a given integer $k$, there are at most two indices $i$ such that $a_i = k$.
\end{enumerate}
The following illustrate, in the case of $\Xpq 45$, how the basis elements can be arranged for various values of $m$.
In each case, the basis element with grading $m(\omega-2) + 2a + 2b\sigma$ is marked by a dot at coordinates $(a,b)$.
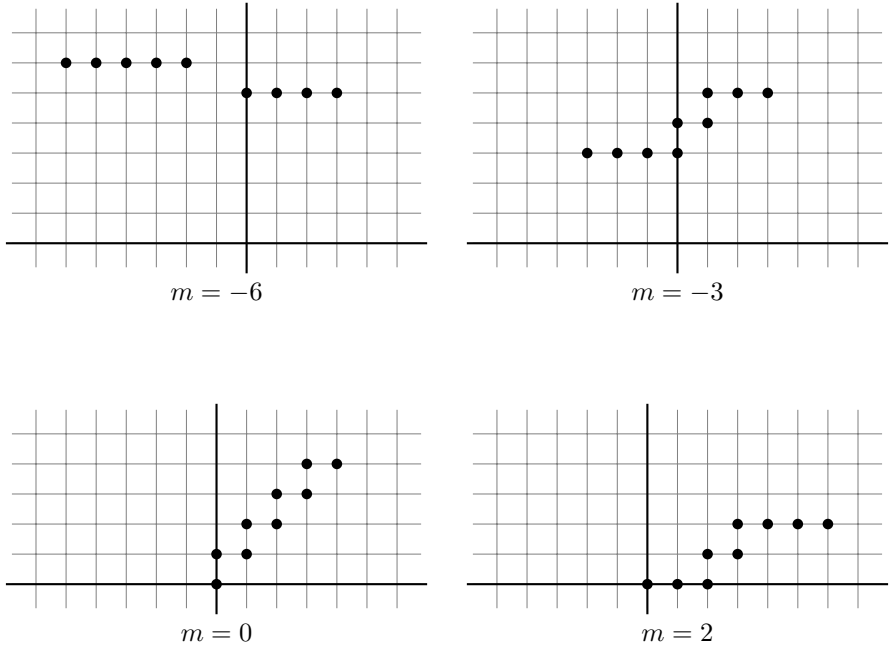
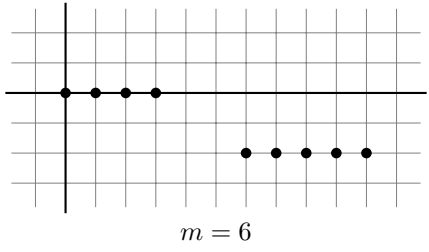
\[
\begin{tikzpicture}[scale=0.4]
	\draw[step=1cm, gray, very thin] (-7.8, -0.8) grid (5.8, 7.8);
	\draw[thick] (-8, 0) -- (6, 0);
	\draw[thick] (0, -1) -- (0, 8);

 	\node[below] at (-1, -1) {$m=-6$};

    \fill (-6, 6) circle(5pt);
    \fill (-5, 6) circle(5pt);
    \fill (-4, 6) circle(5pt);
    \fill (-3, 6) circle(5pt);
    \fill (-2, 6) circle(5pt);

    \fill (0, 5) circle(5pt);
    \fill (1, 5) circle(5pt);
    \fill (2, 5) circle(5pt);
    \fill (3, 5) circle(5pt);

\end{tikzpicture}
\hspace{0.5 cm}
\begin{tikzpicture}[scale=0.4]
	\draw[step=1cm, gray, very thin] (-6.8, -0.8) grid (6.8, 7.8);
	\draw[thick] (-7, 0) -- (7, 0);
	\draw[thick] (0, -1) -- (0, 8);

 	\node[below] at (0, -1) {$m=-3$};

    \fill (-3, 3) circle(5pt);
    \fill (-2, 3) circle(5pt);
    \fill (-1, 3) circle(5pt);
    \fill ( 0, 3) circle(5pt);
    \fill ( 0, 4) circle(5pt);

    \fill (1, 4) circle(5pt);
    \fill (1, 5) circle(5pt);
    \fill (2, 5) circle(5pt);
    \fill (3, 5) circle(5pt);

\end{tikzpicture}
\]
\[
\begin{tikzpicture}[scale=0.4]
	\draw[step=1cm, gray, very thin] (-6.8, -0.8) grid (6.8, 5.8);
	\draw[thick] (-7, 0) -- (7, 0);
	\draw[thick] (0, -1) -- (0, 6);

 	\node[below] at (0, -1) {$m=0$};

    \fill (0, 0) circle(5pt);
    \fill (1, 1) circle(5pt);
    \fill (2, 2) circle(5pt);
    \fill (3, 3) circle(5pt);

    \fill (0, 1) circle(5pt);
    \fill (1, 2) circle(5pt);
    \fill (2, 3) circle(5pt);
    \fill (3, 4) circle(5pt);
    \fill (4, 4) circle(5pt);

\end{tikzpicture}
\hspace{0.5 cm}
\begin{tikzpicture}[scale=0.4]
	\draw[step=1cm, gray, very thin] (-5.8, -0.8) grid (7.8, 5.8);
	\draw[thick] (-6, 0) -- (8, 0);
	\draw[thick] (0, -1) -- (0, 6);

 	\node[below] at (1, -1) {$m=2$};

    \fill (0, 0) circle(5pt);
    \fill (1, 0) circle(5pt);
    \fill (2, 1) circle(5pt);
    \fill (3, 2) circle(5pt);

    \fill (2, 0) circle(5pt);
    \fill (3, 1) circle(5pt);
    \fill (4, 2) circle(5pt);
    \fill (5, 2) circle(5pt);
    \fill (6, 2) circle(5pt);

\end{tikzpicture}
\]
\[
\begin{tikzpicture}[scale=0.4]
	\draw[step=1cm, gray, very thin] (-1.8, -3.8) grid (11.8, 2.8);
	\draw[thick] (-2, 0) -- (12, 0);
	\draw[thick] (0, -4) -- (0, 3);

 	\node[below] at (5, -4) {$m=6$};

    \fill (6, -2) circle(5pt);
    \fill (7, -2) circle(5pt);
    \fill (8, -2) circle(5pt);
    \fill (9, -2) circle(5pt);
    \fill (10, -2) circle(5pt);

    \fill (0, 0) circle(5pt);
    \fill (1, 0) circle(5pt);
    \fill (2, 0) circle(5pt);
    \fill (3, 0) circle(5pt);

\end{tikzpicture}
\]
For ease of reference, we will write the bases as
\[
 F_{p,q}(m) = \{ P^{(m)}_{0}, P^{(m)}_{1}, \ldots, P^{(m)}_{p+q-1} \},
\]
where $P^{(m)}_{i}$ is the basis element in $H_\GG^{m\omega+RO(\GG)}(\Xpq pq_+)$ restricting to the element $\widehat c^{\;i}$ nonequivariantly.
When $m$ is understood, we will simply write $P_i$ for $P^{(m)}_i$.
We can also say that
$P_i$ is the basis element in grading $m(\omega-2)+2a+2b\sigma$ with $a+b = i$, as illustrated in the following diagram
with $m=0$.
\[
\begin{tikzpicture}[scale=0.5]
	\draw[step=1cm, gray, very thin] (-1.8, -0.8) grid (5.8, 5.8);
	\draw[thick] (-2, 0) -- (6, 0);
	\draw[thick] (0, -1) -- (0, 6);

    \fill (0, 0) circle(4pt); \node[below right] at (0,0) {$P_0$};
    \fill (1, 1) circle(4pt); \node[below right] at (1,1) {$P_2$};
    \fill (2, 2) circle(4pt); \node[below right] at (2,2) {$P_4$};
    \fill (3, 3) circle(4pt); \node[below right] at (3,3) {$P_6$};

    \fill (0, 1) circle(4pt); \node[above left] at (0,1) {$P_1$};
    \fill (1, 2) circle(4pt); \node[above left] at (1,2) {$P_3$};
    \fill (2, 3) circle(4pt); \node[above left] at (2,3) {$P_5$};
    \fill (3, 4) circle(4pt); \node[above left] at (3,4) {$P_7$};
    \fill (4, 4) circle(4pt); \node[below right] at (4,4) {$P_8$};

\end{tikzpicture}
\]
\begin{definition}
Given any element $x\in H_\GG^{m\omega+RO(\GG)}(\Xpq pq_+)$, we can write $x$ uniquely as
\[
 x = \sum_{i=0}^{p+q-1} \alpha_i P^{(m)}_{i}
\]
with each coefficient $\alpha_i \in \copt$.
We call the $(p+q)$-tuple $(\alpha_i)$ the {\em coefficient vector} of~$x$.
\end{definition}

Because elements of $\copt$ lie in a restricted set of gradings, the number of nonzero coefficients
possible for a given $x$ may be limited, depending on the grading of $x$, though there are elements $x$ for which all coefficients are nonzero.

There are two restriction maps we will use,
\begin{align*}
    \rho\colon H_\GG^{\alpha}(\Xpq pq_+) &\to H^{|\alpha|}(\Xpq pq_+),
\intertext{restriction to nonequivariant cohomology, and}
    (-)^\GG\colon H_\GG^\alpha(\Xpq pq_+) &\to H^{\alpha_0^\GG}(\Xp p_+) \dirsum H^{\alpha_1^\GG}(\Xq q_+),
\end{align*}
the fixed-point map. These are ring maps and their values on the multiplicative generators are given by the following.
\begin{align*}
    \rho(\cxwt) &= 1 & \rho(\cwt) &= 1 \\
    \rho(\cwd) &= \cd & \rho(\cxwd) &= \cd \\
    \cxwt^\GG &= (0,1) & \cwt^\GG &= (1,0) \\
    \cwd^{\;\GG} &= (\cd, 1) & \cxwd^{\;\GG} &= (1, \cd)
\end{align*}
We also need the values of the similar restriction maps
\begin{align*}
    \rho\colon \copt^\alpha &\to H^{|\alpha|}(S^0), \\
    (-)^\GG\colon \copt^\alpha &\to H^{\alpha^\GG}(S^0).
\end{align*}
The particular values we will need are
\begin{align*}
    \rho(\tau(\iota^{2k})) &= 1 & \rho(e^{-k}\kappa) &= 0 & \rho(e^k) &= 0 \\
    \tau(\iota^{2k})^\GG &= 0 & (e^{-k}\kappa)^\GG &= 2 & (e^k)^\GG &= 1.
\end{align*}

\section{The algebraic equivariant B\'ezout theorem}\label{sec:ideal}


It is possible to take the calculation of Euler classes
in \cite{CHTFiniteProjSpace} and, by brute force, work out their expression
in terms of the basis for the cohomology of $\Xpq pq$ discussed in the preceding section.
Instead, we will take advantage of some features of the cohomology of a point
to give a more conceptual approach that shows better why the calculation works the way it does.

\begin{definition}\ 
\begin{itemize}
\item
Let $T\subset \copt$
consist of the elements
$a\tau(\iota^{2\ell})$ for $a\in\Z$ and $\ell\in\Z$,
$a e^{-m}\kappa$ for $a\in \Z$ and $m\geq 1$,
and the elements $ae^m$ for $a\in\Z$ and $m\geq 1$,
and all of $A(\GG) = \copt^0$.
\item
Let $I_e \subset T$ consist of
the elements $a\tau(\iota^{2\ell})$ for $a\in\Z$ and $\ell\in\Z$,
$a e^{m}\kappa$ for $a\in \Z$ and $m\in \Z$,
and the elements
$a+bg \in A(\GG)$ such that $a$ is even.
\end{itemize}
\end{definition}

Note that $e^m\kappa = 2e^m$ if $m > 0$.

\begin{proposition}
$I_e$ is an ideal of $\copt$.
\end{proposition}

\begin{proof}
This is a straightforward check from the known structure of $\copt$,
as given in~\cite{CHTFiniteProjSpace}.
\end{proof}

On the other hand, $T$ is not an ideal, because $e\xi \notin T$ while $e\in T$.
But $T$ is an additive subgroup.

An important fact about $T$ is that, as shown in the following diagram,
all of its elements lie in gradings of the form $n\sigma$ or $2n(1+\sigma)$,
that is, on the vertical line through the origin or the diagonal through the origin with slope $-1$.
Closed circles indicate copies of $\Z$,
while the box at the origin is $A(\GG)$.
$T$ is a free $\Z$-module.
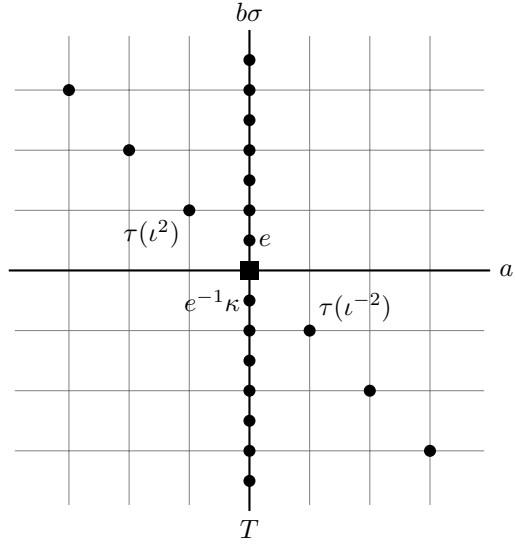
\[
\begin{tikzpicture}[x=4mm, y=4mm]
	\draw[step=2, gray, very thin] (-7.8, -7.8) grid (7.8, 7.8);
	\draw[thick] (-8, 0) -- (8, 0);
	\draw[thick] (0, -8) -- (0, 8);
    \node[right] at (8,0) {$a$};
    \node[above] at (0,8) {$b\sigma$};
    \node[below] at (0,-8) {$T$};

    \fill (-0.3, -0.3) rectangle (0.3, 0.3);
    \fill (0, -7) circle(0.2);
    \fill (0, -6) circle(0.2);
    \fill (0, -5) circle(0.2);
    \fill (0, -4) circle(0.2);
    \fill (0, -3) circle(0.2);
    \fill (0, -2) circle(0.2);
    \fill (0, -1) circle(0.2);
    \fill (0, 1) circle(0.2);
    \fill (0, 2) circle(0.2);
    \fill (0, 3) circle(0.2);
    \fill (0, 4) circle(0.2);
    \fill (0, 5) circle(0.2);
    \fill (0, 6) circle(0.2);
    \fill (0, 7) circle(0.2);

    \fill (-2, 2) circle(0.2);
    \fill (-4, 4) circle(0.2);
    \fill (-6, 6) circle(0.2);

    \fill (2, -2) circle (0.2);
    \fill (4, -4) circle (0.2);
    \fill (6, -6) circle (0.2);
    
     \node[right] at (0,1) {$e$};
    \node[below left] at (-2,2) {$\tau(\iota^2)$};
    \node[above right] at (2,-2) {$\tau(\iota^{-2})$};
    \node[left] at (0,-1) {$e^{-1}\kappa$};

\end{tikzpicture}
\]

Another fact that follows from the known structure of $\copt$ is that
the quotient ring $\copt/I_e$ is all 2-torsion.

\begin{remark}
The ideal $I_e$ is almost, but not quite, the kernel of the restriction map
$\copt = \HA(S^0;\Mackey A)\to \HA(S^0;\Mackey{\Z/2})$.
That kernel would not contain all the elements $a\tau(\iota^{-2m})$ for $m\geq 1$, but only
those of the form $2a\tau(\iota^{-2m})$. Either ideal would serve our purpose here,
but we chose to use the one that is slightly simpler to describe.
\end{remark}

\begin{definition}\ 
\begin{itemize}
\item
Let 
$\tilde T \subset H_\GG^{RO(\Pi B)}(\Xpq pq_+)$
denote the set of linear combinations of elements of our preferred basis of
$H_\GG^{RO(\Pi B)}(\Xpq pq_+)$, with coefficients in $T$.
\item
Let $J_e$ be the ideal defined by
\[
    J_e = I_e\cdot H_\GG^{RO(\Pi B)}(\Xpq pq_+) \subset H_\GG^{RO(\Pi B)}(\Xpq pq_+).
\]
\end{itemize}
\end{definition}

Every element of $J_e$ is a linear combination of elements from our preferred basis
with coefficients in $I_e$ (and this would be true for any basis we used).
Because $J_e\subset \tilde T$, the following facts about $\tilde T$ apply to $J_e$ as well.

\begin{lemma}\label{lem:triple}
Every element $x\in \tilde T$ is a linear combination of at most three basis elements:
If $x$ lies in grading $m(\omega-2) + a + b\sigma$, 
the only basis elements that can contribute to $x$ are the one (if any) lying on the
same diagonal as $x$, that is, in a grading $m(\omega-2) + a' + b'\sigma$ with $a'+b' = a+b$,
and the two (at most) lying in the same vertical line as $x$, that is,
in gradings $m(\omega-2) + a + b'\sigma$.
\end{lemma}

\begin{proof}
This follows from the description of the locations of the basis elements
given in the preceding section together with the locations of the elements of $T$.
\end{proof}

See the example in Remark~\ref{rem:example} below for an illustration of this lemma.


\begin{proposition}\label{prop:restrictionsdetermine}
If $x\in \tilde T$, then $x$ is determined by its restrictions $\rho(x)$ and $x^\GG$.
\end{proposition}

\begin{proof}
By the preceding lemma, $x$ can be written as a linear combination of
at most three elements from our standard basis. There are various cases that
should be considered. Suppose, for example, that $x$ lies on the same diagonal
as a basis element $P_n$ and lies above two basis elements $P_k$ and $P_{k-1}$.
Then we can write
\[
    x = \alpha\tau(\iota^{2\ell})P_n + \beta e^m P_k + \gamma e^{m+2} P_{k-1}
\]
for some integers $\alpha$, $\beta$, $\gamma$, $\ell$, and $m$.
We now appeal to \cite[4.6]{CHTFiniteProjSpace}, where we showed that our standard basis
restricts to a nonequivariant basis for $\Xpq pq$ and a nonequivariant basis for $\Xpq pq^\GG$.
We have $\rho(x) = 2\alpha\rho(P_n)$, so $\alpha$ is determined by $\rho(x)$.
On the other hand, $x^\GG = \beta P_k^\GG + \gamma P_{k-1}^\GG$, so $\beta$ and $\gamma$
are determined by $x^\GG$.

There are other cases, for example, where $x$ lies below two basis elements rather than above,
or where it lies in the same grading as a basis element. Each of these cases can be handled
in the same way as the case above.
\end{proof}

Note that this is not true for general elements of $H_\GG^{RO(\Pi B)}(\Xpq pq_+)$
because there are elements of $\copt$ that vanish under both $\rho$ and $(-)^\GG$.

For any $x\in H_\GG^{RO(\Pi B)}(\Xpq pq_+)$, we have
\[
    \rho(x) \in H^\Z(\Xp{p+q}_+),
\]
so $\rho(x) = \Delta\cd[k]$ for some integers $\Delta$ and $k$, or is 0, in which case we set $\Delta = 0$. 
We also have
\[
    x^\GG \in H^\Z(\Xp p_+) \dirsum H^\Z(\Xp q_+),
\]
so $x^\GG = (\Delta_0 \cd[i], \Delta_1\cd[j])$ for some integers $\Delta_0$, $\Delta_1$, $i$, and $j$.
(Again, we set $\Delta_0 = 0$ if $\Delta_0\cd[i] = 0$ and $\Delta_1 = 0$ if $\Delta_1\cd[j] = 0$.)

\begin{definition}
We call the triple of integers $(\Delta,\Delta_0,\Delta_1)$ determined as above the {\em $\GG$-degrees} of $x$.
\end{definition}

\begin{corollary}\label{cor:degreesdetermine}
If $x\in\tilde T$, then $x$ is determined by its grading and its $\GG$-degrees.
\end{corollary}

\begin{proof}
Suppose that $x$ lies in grading $m(\omega-2) + a + b\sigma$
and that the degrees of $x$ are $(\Delta,\Delta_0,\Delta_1)$.
By the structure of $\tilde T$ and the locations of the basis elements,
we can assume that $a$ is even.
Then we have
\[
    \rho(x) = 
    \begin{cases}
        \Delta\cd[(a+b)/2] & \text{if $b$ is even} \\
        0 & \text{otherwise}
    \end{cases}
\]
and
\[
    x^\GG = (\Delta_0\cd[a/2], \Delta_1\cd[a/2-m])
\]
Thus, the grading of $x$ and its degrees determine $\rho(x)$ and $x^\GG$, so the
result follows from the preceding proposition.
\end{proof}




In order to apply these results to derive the two parts of B\'ezout's theorem, we
need to know a little more about the line bundles that are the summands of $F$ as
in the B\'ezout context~\ref{context}.
In \cite{CHTFiniteProjSpace} we showed that
the line bundles over $\Xpq pq$ all have the form $O(d)$ or $\chi O(d)$.
It is useful to further break these down into four types:
\begin{align*}
    \text{Type I: } & \text{bundles of the form $O(2d+1)$} \\
    \text{Type II: } & \text{bundles of the form $O(2d)$} \\
    \text{Type III: } & \text{bundles of the form $\chi O(2d+1)$} \\
    \text{Type IV: } & \text{bundles of the form $\chi O(2d)$}    
\end{align*}
The fixed points $O(2d+1)^\GG$ of a bundle of type I have fiber $\C$ over $\Xp p$ and $0$ over $\Xq q$,
while the reverse is true for a bundle of type III.
The fixed points $O(2d)^\GG$ of a bundle of type II have fiber $\C$ over both components of $\Xpq pq^\GG$,
while the fixed points of a bundle of type IV have fiber $0$ over both components.

In \cite{CHTFiniteProjSpace} for $\dagger\in\{\text{I,\,II,\,III,\,IV}\}$ we wrote $n_\dagger$ for the number of summands of type $\dagger$ and $d_\dagger$ for the products
of their degrees. 
These are related to the ranks and $\GG$-degrees of $F$ by
\begin{align}
    n &= n_\I + n_\II + n_\III + n_\IV \notag\\
    n_0 &= n_\I + n_\II \notag\\
    n_1 &= n_\II + n_\III \notag \\
    \Delta &= d_\I d_\II d_\III d_\IV \notag\\
    \Delta_0 &= 
        \begin{cases}
            d_\I d_\II & \text{if\ \,$n_0 < p$} \\
            0 & \text{if\,\  $n_0 \geq p$}
        \end{cases} \label{Delta_0}\\
    \Delta_1 &= 
        \begin{cases}
            d_\II d_\III & \text{if\ \,$n_1 < q$} \\
            0 & \text{if\,\  $n_1 \geq q$.}
        \end{cases}\label{Delta_1}
\end{align}
Now, $d_\I$ and $d_\III$ are always odd, and $d_\II$ and $d_\IV$ are even if and only if
there is a summand of type II or IV, respectively.
Notice that, when $n_\II > 0$, the quantities $\Delta$, $\Delta_0$, and $\Delta_1$ will all be even.
If $n_\II = 0$, then $n_0 + n_1 \leq n$, which implies that
\[
    n_0 \leq n - n_1 \leq n - (n-p) = p
\]
and $n_1 \leq q$, similarly, with equality possible only if $n_\IV = 0$.
So, if $n_\II = 0$ but $n_\IV > 0$, we will have $\Delta$ even
and both $\Delta_0$ and $\Delta_1$ odd.
When $n_\II = 0$ and $n_\IV = 0$, we will have $\Delta$ odd
while $\Delta_0$ and $\Delta_1$ will be odd if nonzero.

\begin{theorem}[B\'ezout's Theorem, Part I]\label{thm:bezoutone}
Let $F$ be as in the B\'ezout context~\ref{context}.
Then $e(F)$ lies in $\tilde T$, hence is determined
by its grading, which is
\[
    (n_0 - n_1)(\omega-2) + 2n_0 + 2(n-n_0)\sigma,
\]
and its $\GG$-degrees, which are $(\Delta,\Delta_0,\Delta_1)$.
Moreover, the grading and degrees can be recovered from $e(F)$.
The ranks $(n,n_0,n_1)$ are additive while the degrees are multiplicative.
\end{theorem}

\begin{proof}
The additivity of the grading and the multiplicativity of the degrees are clear
(but see the caveat about multiplicativity given in the Introduction).

Given that $n$ is the nonequivariant (complex) rank of $F$ and
$n_0$ and $n_1$ are the ranks of the restriction of $F^\GG$ to $\Xp p$ and $\Xq q$, respectively,
$e(F)$ must lie in the grading given, which is the grading $\alpha$ with
$|\alpha| = n$, $\alpha_0 = 2n_0 + 2(n-n_0)\sigma$, and $\alpha_1 = n_1 + 2(n-n_1)\sigma$.

Conversely, if $e(F)$ lies in grading $m(\omega-2) + 2a + 2b\sigma$, then we can recover
$n = a + b$, $n_0 = a$, and $n_1 = a - m$.

The degrees $(\Delta,\Delta_0,\Delta_1)$ are, by the nonequivariant B\'ezout theorem, given by
\begin{align*}
    \rho(e(F)) &= \Delta \cd[n] \\
    e(F)^\GG &= ( \Delta_0\cd[n_0], \Delta_1\cd[n_1]),
\end{align*}
using the fact that $\rho$ and $(-)^\GG$ preserve Euler classes.
Thus, we can recover the degrees from $e(F)$.

It remains to show that $e(F)$ is determined by its grading and $\GG$-degrees.

Recall the discussion before the theorem of the four types of line bundles over $\Xpq pq$.
In \cite[Proposition 6.5]{CHTFiniteProjSpace} we computed their Euler classes, which are
\begin{alignat*}{3}
    e(O(2d+1)) &= \cwd + d(\tau(1)\cwd + e^{-2}\kappa\cwt\cwd\cxwd) &&\equiv \cwd &&\pmod{J_e} \\
    e(O(2d)) &= d(\tau(\iota^{-2})\cxwt\cwd + e^{-2}\kappa \cwd\cxwd) &&\equiv 0 &&\pmod{J_e} \\
    e(\chi O(2d+1)) &= \cxwd + d(\tau(1)\cxwd + e^{-2}\kappa\cxwt\cwd\cxwd) &&\equiv \cxwd &&\pmod{J_e} \\
    e(\chi O(2d)) &= e^2 + d\tau(1)\cxwt\cwd &&\equiv e^2 &&\pmod{J_e}.
\end{alignat*}
From (\ref{Delta_0}) and (\ref{Delta_1}),
we see that $\Delta_0$  and $\Delta_1$ are both even if and only if $F$ contains at least one summand of the form $O(2d)$ (type II).
If $F$ does not contain such a summand, then $n_0$ is the number of summands of the form $O(2d+1)$
and $n_1$ is the number of summands of the form $\chi O(2d+1)$, and we will have $n_0 + n_1 \leq n$.
From the congruences above we have, modulo $J_e$, that
\[
    e(F) \equiv
    \begin{cases}
        0 & \text{if $\Delta_0$ and $\Delta_1$ are even} \\
        e^{2(n - n_0 - n_1)}\cwd^{\;n_0}\cxwd^{\;n_1} & \text{if $\Delta_0$ or $\Delta_1$ is odd}.    \end{cases}
\]
When $\Delta_0$ or $\Delta_1$ is odd we have that 
$n_0 \leq p$ and $n_1 \leq q$, with at least one of the inequalities being strict,
so $\cwd^{\;n_0}\cxwd^{\;n_1}$ is a basis element and 
$e^{2(n - n_0 - n_1)}\cwd^{\;n_0}\cxwd^{\;n_1} \in\tilde T$.
It follows that $e(F)\in\tilde T$,
and then the fact that $e(F)$ is determined by its grading and $\GG$-degrees
follows from Corollary~\ref{cor:degreesdetermine}.
\end{proof}




By Lemma~\ref{lem:triple}, the Euler class $e(F)$ can be written as a linear combination
of just three basis elements.
We next work out the explicit expression, which, by Theorem~\ref{thm:bezoutone},
is determined by the grading of
$e(F)$ and its $\GG$-degrees.

\begin{theorem}[B\'ezout's Theorem, Part II]\label{thm:bezouttwo}
Let $F$ be as in the B\'ezout context~\ref{context}.
Then we can write
\[
    e(F) = \alpha P^{(m)}_{n} + \beta P^{(m)}_{k} + \gamma P^{(m)}_{k-1}
\]
for some $1\leq k < p+q$ and some coefficients $\alpha$, $\beta$, and $\gamma$ in $\copt$,
so the coefficient vector of $e(F)$ has at most three nonzero components.
Allowing for the possibility that $n = k$ or $n = k-1$, 
we can arrange that the coefficient $\alpha$
is always an integer multiple of $\tau(\iota^{2i})$ for some $i\in\Z$, and
the coefficients $\beta$ and $\gamma$ are always integer multiples of
$e^{2i}$ or $e^{-2i}\kappa$ for some $i\geq 0$.

Use the briefer notation $P_{n}$ and write $\epsilon = 0$ or $1$ for the remainder on dividing
$n+n_0+n_1$ by $2$. We have
\begin{align*}
 P_{n} &=
   \begin{cases}
 	\cxwt^{-(n+n_0-n_1-2p)} \cwd^{\;p}\cxwd^{\;n-p}
		& \text{if\,\ $n+n_0-n_1 > 2p$} \\
	\cwt^{-(n-n_0+n_1-2q)} \cwd^{\;n-q}\cxwd^{\;q}
		& \text{if\,\ $n-n_0+n_1 > 2q$} \\
	\cxwt^\epsilon \cwd^{\;(n+n_0-n_1+\epsilon)/2}
			\cxwd^{\;(n-n_0+n_1-\epsilon)/2}
		& \text{otherwise,}
   \end{cases}
\\
 P_k &=
    \begin{cases}
        \cxwt\cwd^{\;n_0+1}\cxwd^{\;n_1} & \text{if\,\ $n_0 < p$} \\
        \cxwt^{\;-(n_0-p)}\cwd^{\;p}\cxwd^{\;n_1} & \text{if\,\  $n_0\geq p$,}
    \end{cases}
\\ \intertext{and}
 P_{k-1} &=
    \begin{cases}
        \cwd^{\;n_0}\cxwd^{\;n_1} & \text{if\,\ $n_1 < q$} \\
        \cwt^{\;-(n_1-q)}\cwd^{\;n_0}\cxwd^{\;q} & \text{ if\,\ $n_1\geq q$.}
    \end{cases}
\end{align*}
The coefficient $\alpha$ will be an integer multiple of
\[
 \tau_{n} = 
   \begin{cases}
 	\tau(\iota^{2(n-n_1-p)})
		& \text{if\,\ $n+n_0-n_1 > 2p$} \\
	\tau(\iota^{2(n-n_0-q)})
		& \text{if\,\ $n-n_0+n_1 > 2q$} \\
	\tau(\iota^{n-n_0-n_1-\epsilon})
		& \text{otherwise.}
   \end{cases}
\]
Finally, write $\bar n_0 = \min\{n_0,p-1\}$ and $\bar n_1 = \min\{n_1,q\}$.
Then, we break the result into the following cases.
\begin{enumerate}

\item If $\Delta$ is even, then
\begin{align*}
 \alpha &= \frac{\Delta}{2}\tau_{n}, 
 & \beta &= \frac{\Delta_1-\Delta_0}{2} e^{-2(\bar n_0+\bar n_1-n+1)}\kappa, \\ 
 \gamma &= \frac{\Delta_0}{2} e^{-2(\bar n_0+\bar n_1-n)}\kappa
 & \text{and}\quad k &= \bar n_0 + \bar n_1 + 1.
\end{align*}

\item If $\Delta$ is odd and $\Delta_0\neq 0$, then
\begin{align*}
 \alpha &= \frac{\Delta - \Delta_0}{2}\tau(1)
 & \beta &= \dfrac{\Delta_1 - \Delta_0}{2} e^{-2}\kappa, \\ 
 \gamma &=  \Delta_0
 & \text{and}\quad k &= n+1.
\end{align*}

\item If $\Delta$ is odd and $\Delta_0 = 0$, then
\begin{align*}
 \alpha &= \frac{\Delta - \Delta_1}{2}\tau(1)
 & \beta &= 0, \\ 
 \gamma &=  \Delta_1
 & \text{and}\quad k &= n+1.
\end{align*}
\end{enumerate}

\end{theorem}

\begin{proof}


Theorem~\ref{thm:bezoutone} and Lemma~\ref{lem:triple} imply the first claim, that we can write
$e(F)$ in terms of just three basis elements. 

To determine $P_n$, $P_k$, and $P_{k-1}$, we recall from \cite[Proposition~4.7]{CHTFiniteProjSpace}
that the basis elements take one of six possible forms:
\begin{align*}
    \cwt^m\cwd^{\;a}& \quad m > 1,\ a< p &  \cxwt^m\cxwd^{\;b} & \quad m > 1,\ b < q \\
    \cwd^{\;a}\cxwd^{\;b} & \quad a \leq p,\ b\leq q
        & \cxwt  \cwd^{\;a}\cxwd^{\;b} & \quad a \leq p,\ b < q \\
    \cxwt^{-m}\cwd^{\;p}\cxwd^{\;b} & \quad m > 0,\ b < q
        & \cwt^{-m}\cwd^{\;a}\cxwd^{\;q} & \quad m > 0,\ a < p\\
\end{align*}
where we recall that $\cwd^{\;p}\cxwd^{\;q} = 0$, so we do not have $a = p$ and $b=q$ above.

We noted earlier that $e(F)$ lies in grading
\[
    \grad e(F) = (n_0 - n_1)(\omega - 2 ) + 2n_0 + 2(n-n_0)\sigma.
\]
$P_n$ is the unique basis element having grading in $(n_0-n_1)(\omega-2) + RO(\GG)$
restricting to $\cd[n]$, and we can check that the formula given in the statement
of the theorem has those properties.
Similarly, $P_k$ and $P_{k-1}$ are the (at most) two basis elements having gradings of the
form $(n_0-n_1)(\omega-2) + 2n_0 + 2b\sigma$
and we can check that the formulas given have that property.
The coefficient $\tau_{n}$ is the element of the form $\tau(\iota^{2i})$ such that
$\tau_{n} P_{n}$ lies in the same grading as $e(F)$.
The terms of the form $e^m\kappa$ multiplying $P_k$ and $P_{k-1}$
in the formulas for $e(F)$ are determined similarly.

We should point out some abuses of notation we are indulging in.
The formulas for $P_k$ and $P_{k-1}$ evaluate to 0, not basis elements, 
when both $n_0\geq p$ and $n_1\geq q$.
In the case $n_0 < p-1$ and $n_1\geq q$, the formula for $P_k$ is not
a basis element, but we know that its coefficient will be a multiple of $e^m\kappa$ for some
integer $m$, and the product $e^m\kappa P_k = 0$ in that case, because of the relations in
the cohomology of $\Xpq pq$.
A similar vanishing happens in the case of $P_{k-1}$ when $n_0 > p$ and $n_1 < q$.
Finally, the formulas for $P_k$ and $P_{k-1}$ coincide when $n_0 = p$ and $n_1 < q$,
but in that case $\Delta_0 = 0$ so only one copy of this basis element appears in the formula
for $e(F)$.

To verify the coefficients of $P_n$, $P_k$, and $P_{k-1}$, we use the fact that $e(F)$
is determined by the nonequivariant elements
\begin{align*}
    \rho(e(F)) &= \Delta\cd[n] \\
    \mathllap{\text{and}\qquad} e(F)^\GG &= (\Delta_0\cd[n_0], \Delta_1\cd[n_1]),
\end{align*}
so we simply need to check that the formulas of the theorem have the correct values
on applying these restriction maps.

First note that, regardless of which case we fall in, we will always have
\begin{align*}
    \rho(\tau_n P_n) &= 2\cd[n] \\
    (\tau_n P_n)^\GG &= (0,0).
\end{align*}
For $P_k$ and $P_{k-1}$ we have
\begin{align*}
    \rho(P_k) &= \cd[k] & \rho(P_{k-1}) &= \cd[k-1] \\
    P_k^\GG &= (0, \cd[n_1])
    & P_{k-1} &= (\cd[n_0], \cd[n_1])
\end{align*}
which includes the possibility that $P_{k-1}^\GG = (\cd[n_0],0)$ if $n_1 \geq q$.

Now, when $\Delta$ is even, in the formulas given, $\beta$ and $\gamma$ each have a factor of
the form $e^{m}\kappa$, and we have $\rho(e^m\kappa) = 0$ and $(e^m\kappa)^\GG = 2$.
Combined with the formulas above, this verifies case (1) of the theorem,
except that we should say something about the parities of $\Delta_0$ and $\Delta_1$.
From the discussion before Theorem~\ref{thm:bezoutone}, because $\Delta$ is even,
$\Delta_0$ and $\Delta_1$ have the same parity.
There is a possibility that $\Delta_0$ is odd, but this can
happen only when $n_\II = 0$ and $n_\IV > 0$, in which case $n_0 < p$, $n_1 < q$,
and $n_0 + n_1 < n$. The coefficient $\gamma$ in that case is
\[
    \gamma = \frac{\Delta_0}2 e^{-2(n_0+n_1-n)}\kappa = \frac{\Delta_0}2 2e^{2(n-n_0-n_1)}
\]
which we interpret as $\Delta_0 e^{2(n-n_0-n_1)}$ by another abuse of notation.
(The abuse is that division by 2 is not well-defined in $\copt$.)
We then use that $\rho(e^m) = 0$ and $(e^m)^\GG = 1$ for $m>0$.

If $\Delta$ is odd, then $n = n_0 + n_1$, $n_0\leq p$, and $n_1\leq q$.
IF $\Delta_0$ and $\Delta_1$ are both nonzero, then $n_0 < p$ and $n_1 < q$,
$P_n = P_{k-1}$, and the formula in case (2) of the theorem is easily verified.

If $\Delta_0\neq 0$ but $\Delta_1 = 0$, then we have $n_0 < p$ and $n_1 = q$.
In this case, we have
\[
    e^{-2}\kappa P_k = e^{-2}\kappa\cxwt\cwd^{\;n_0+1}\cxwd^{\;q} = 0,
\]
so we allow the abuse of notation that $\Delta_1 - \Delta_0$ is odd in the formula for $\beta$.
With that caveat, the verification of case (2) can be completed.

In Case (3), since $\Delta_0 = 0$ we must have $\Delta_1 \neq 0$ and odd.
The verification is then just as for the previous cases.

The asymmetries in these formulas comes
from an asymmetry in our preferred basis regarding $\cwd$ vs $\cxwd$.
\end{proof}

\begin{remark}
Theorems~\ref{thm:bezoutone} and~\ref{thm:bezouttwo} give us two related ways
of determining $e(F)$: 
It is determined by the ranks $(n,n_0,n_1)$ and the $\GG$-degrees $(\Delta,\Delta_0,\Delta_1)$,
and it is also determined by its triple of nonzero coefficients.
The advantage of using the degrees is that they are multiplicative.
This is simpler to calculate with and also parallels the
result of the nonequivariant B\'ezout theorem, that degrees are multiplicative
under intersection of projective varieties.
\end{remark}

\begin{remark}\label{rem:example}
The summary of Theorem~\ref{thm:bezouttwo} is that $e(F)$
can be expressed in terms of at most three basis elements.
This is not a restriction imposed by the locations of the basis elements.
As an example, consider $\Xpq 55$ and the bundle $F = 4\chi O(2)$, the sum of 4 copies of $\chi O(2)$,
so $n = 4$ and $n_0 = n_1 = 0$.
This Euler class lives in grading
\[
    (n_0-n_1)(\omega-2) + 2n_0 + 2(n-n_0)\sigma = 8\sigma.
\]
The following diagram shows the location of $e(F)$, the ``$\times$'' at $8\sigma$, and the locations of the basis elements in
the $RO(\GG)$-grading:
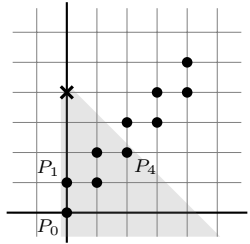
\[
\begin{tikzpicture}[scale=0.4]
	\draw[step=1cm, gray, very thin] (-1.8, -0.8) grid (5.8, 6.8);
	\draw[thick] (-2, 0) -- (6, 0);
	\draw[thick] (0, -1) -- (0, 7);

    \fill[gray, opacity=0.2] (-0.2, 3.8) -- (0.15, 4.15) -- (5.1, -0.8) -- (-0.2, -0.8) -- cycle;

    \fill (0, 0) circle(5pt); \node[below left=-1pt] at (0,0) {$\scriptstyle P_0$};
    \fill (1, 1) circle(5pt);
    \fill (2, 2) circle(5pt); \node[below right=-1pt] at (2,2) {$\scriptstyle P_4$};
    \fill (3, 3) circle(5pt);
    \fill (4, 4) circle(5pt);

    \fill (0, 1) circle(5pt); \node[above left=-1pt] at (0,1) {$\scriptstyle P_1$};
    \fill (1, 2) circle(5pt);
    \fill (2, 3) circle(5pt);
    \fill (3, 4) circle(5pt);
    \fill (4, 5) circle(5pt);

    \draw[very thick] (-0.2,3.8) -- (0.2,4.2);
    \draw[very thick] (-0.2,4.2) -- (0.2,3.8);

\end{tikzpicture}
\]
The five basis elements within the shaded area have nonzero multiples in degree $8\sigma$,
so could conceivably contribute to $e(F)$,
but the theorem says that it can be written in terms of just three of them,
the one on the same diagonal as $e(F)$, $P_4$, and the two below it, $P_0$ and $P_1$.

In fact, we are in case~(1) of Theorem~\ref{thm:bezouttwo}, with $\Delta = 8$ and $\Delta_0 = \Delta_1 = 1$, so
\[
    e(4\chi O(2)) = 8\tau(\iota^4)P_4 + 0\cdot P_1 + e^8 P_0
        = 8\tau(\iota^4)\cwd^{\;2}\cxwd^{\;2} + e^8.
\]
As it happens, $P_1$ does not actually contribute in this example.
\end{remark}

\begin{remark}
In \cite{CHTFiniteProjSpace}, we looked in detail at the case $n = p+q-1$, so that the
hypersurfaces associated with the line bundle summands of $F$ intersect generically in a $\GG$-set of points
in $\Xpq pq$.
In that case, we showed that the explicit formula for $e(F)$ can be read as telling us how that collection of points
breaks down as free orbits versus fixed points in each of the components of $\Xpq pq^\GG$.
In a followup to this paper, we will show how the Euler class more generally 
gives us geometric information about the intersection of hypersurfaces.
\end{remark}

\section{Comparison with constant $\Z$ coefficients}\label{sec:Z}

Another equivariant cohomology theory commonly used is ordinary cohomology with
coefficients in $\Mackey \Z$, the constant-$\Z$ Mackey functor.
We calculate the Euler class $e(F)$ with $\Mackey\Z$ coefficients and compare
it to the class obtained with Burnside ring coefficients.

As shown in \cite{Co:BGU1preprint}, $H_\GG^{RO(\GG)}(S^0;\Mackey\Z)$ is obtained from
$\copt$ by setting $\kappa = 0$. This has the effect
of removing the elements $e^{-n}\kappa$ and making $2e = 0$.
Since $\kappa = 2-g$, it also has the effect of setting $g = 2$.
Put another way, this theory cannot distinguish between a free orbit and two fixed points.

Because the cohomology of $\Xpq pq$ with $\Mackey A$ coefficients is free over the cohomology of a point,
we obtain the cohomology with $\Mackey\Z$ coefficients by setting $\kappa = 0$.
The result is the following.

\begin{theorem}[{\cite[Corollary~5.4]{CHTFiniteProjSpace}}]\label{thm:ZcohomStructure}
Let $0 \leq p < \infty$ and $0 \leq q < \infty$ with $p+q > 0$.
Then $H_{\GG}^{RO(\Pi B)}(\Xpq{p}{q}_+;\Mackey\Z)$ is a free module over $H_{\GG}^{RO(\GG)}(S^0;\Mackey\Z)$.
Its structure as a graded commutative algebra over $H_{\GG}^{RO(\GG)}(S^0;\Mackey\Z)$
is described as in Theorem~\ref{thm:cohomStructure}, except that
the relation $\cwt \cxwd = (1-\kappa)\cxwt \cwd + e^2$ is replaced by the relation
\[
    \cwt \cxwd = \cxwt \cwd + e^2.
\]
\qed
\end{theorem}


Setting $\kappa=0$ in Theorem~\ref{thm:bezouttwo}, remembering that $e^m$ is 2-torsion, and paying attention to the abuses of notation
mentioned in the proof of that theorem, we get the following.

\begin{theorem}[B\'ezout's Theorem for Constant $\Z$ Coefficients]\label{thm:Zalgebraic}
Let $F$ be as in the B\'ezout context~\ref{context}.
Then the Euler class $e_\Z(F)\in H_\GG^{RO(\Pi B)}(\Xpq pq_+;\Mackey\Z)$ is given by
\[
    e_\Z(F) = 
        \begin{cases}
            \dfrac \Delta 2\tau_{n} P^{(m)}_{n} & \text{if\ \, $\Delta$, $\Delta_0$, and $\Delta_1$} \\
                & \text{\quad are even} \\
            \dfrac \Delta 2\tau_{n} P^{(m)}_{n} + e^{2(n-n_0-n_1)} P^{(m)}_{k-1}
                & \text{if\ \,$\Delta$ is even and} \\
                & \text{\quad  $\Delta_0$ or $\Delta_1$ is odd} \\
            \Delta P^{(m)}_{n} & \text{if\,\ $\Delta$ is odd}
        \end{cases}
\]
where, writing $\epsilon = 0$ or $1$ for the remainder on dividing
$n+n_0+n_1$ by $2$, we set
\begin{align*}
 P_{n}^{(m)} &=
   \begin{cases}
 	\cxwt^{-(n+n_0-n_1-2p)} \cwd^{\;p}\cxwd^{\;n-p}
		& \text{if\,\ $n+n_0-n_1 > 2p$} \\
	\cwt^{-(n-n_0+n_1-2q)} \cwd^{\;n-q}\cxwd^{\;q}
		& \text{if\,\ $n-n_0+n_1 > 2q$} \\
	\cxwt^\epsilon \cwd^{\;(n+n_0-n_1+\epsilon)/2}
			\cxwd^{\;(n-n_0+n_1-\epsilon)/2}
		& \text{otherwise.}
   \end{cases} 
\\
 \tau_{n} &= 
   \begin{cases}
 	\tau(\iota^{2(n-n_1-p)})
		& \text{if\,\ $n+n_0-n_1 > 2p$} \\
	\tau(\iota^{2(n-n_0-q)})
		& \text{if\,\ $n-n_0+n_1 > 2q$} \\
	\tau(\iota^{n-n_0-n_1-\epsilon})
		& \text{otherwise.}
   \end{cases}
\\ \intertext{and, when $\Delta$ is even and $\Delta_0$ or $\Delta_1$ is odd,}
   P^{(m)}_{k-1} &= \cwd^{\;n_0}\cxwd^{\;n_1}.
\end{align*}
\qed
\end{theorem}


While this result has the benefit of relative simplicity, it carries significantly less information than Theorem~\ref{thm:bezouttwo}.
In particular, we cannot reconstruct $\Delta_0$ and $\Delta_1$ from $e_\Z(F)$.
This follows from the formula in the theorem, but we can also look
again at the fixed-point map $(-)^\GG$ to see why this must happen.
As defined in \cite{CostenobleWanerBook}, the fixed-point map
takes $G$-equivariant cohomology with coefficients in a Mackey functor $\Mackey T$
to nonequivariant cohomology with coefficients in $\Mackey T^G$.
In the case of the group $\GG$, we have
\[
 \Mackey T^\GG = \Mackey T(\GG/\GG)/\tau(\Mackey T(\GG/e)).
\]
This gives $\Mackey A^\GG = \Z$, but $\Mackey\Z^\GG = \Z/2$.
We then get the following.

\begin{corollary}\label{cor:Zrestrictions}
With $F$ as in the B\'ezout context~\ref{context},
we have
\[
    e_\Z(F)^\GG = ( \Delta_0 \widehat c^{\;n_0}, \Delta_1 \widehat c^{\;n_1} ) \\
    \in H^{2a}(\Xp p_+;\Z/2) \dirsum H^{2(a-m)}(\Xq q_+;\Z/2),
\]
so
\[
    e_\Z(F)^\GG =
    \begin{cases}
        (0,0) & \text{if $\Delta_0$ and $\Delta_1$ are even} \\
        ( \widehat c^{\;n_0},  \widehat c^{\;n_1} ) & \text{if $\Delta_0$ or $\Delta_1$ is odd.}
    \end{cases}
\]
\end{corollary}

\begin{proof}
From the commutativity of the diagram
\[
 \xymatrix{
    H_\GG^{RO(\Pi B)}(\Xpq pq_+;\Mackey A) \ar[d]_{(-)^\GG} \ar[r]
        & H_\GG^{RO(\Pi B)}(\Xpq pq_+;\Mackey\Z) \ar[d]^{(-)^\GG} \\
    H^\Z(\Xpq pq_+^\GG;\Z) \ar[r] & H^\Z(\Xpq pq_+^\GG;\Z/2),
 }
\]
where the horizontal arrows are given by change of coefficients,
$e_\Z(F)^\GG$ is just the reduction of $e(F)^\GG$ modulo 2.
\end{proof}

Thus, from this Euler class we cannot recover $\Delta_0$ and $\Delta_1$, only their parities.
This goes back to the fact that, because $g = 2$, cohomology with $\Mackey\Z$ coefficients
cannot distinguish between a free orbit and two fixed points, hence retains only parity information
about fixed points.

For example,
in the case $n = p+q-1$ discussed in detail in \cite{CHTFiniteProjSpace},
we can think of the Euler class in terms of the finite $\GG$-set given by
the zero locus of a section of $F$, or the intersection of the hypersurfaces given by
the zero loci of sections of the line bundles making up $F$.
The Euler class with Burnside ring coefficients completely determines this $\GG$-set,
including how many fixed points lie in each component of $\Xpq pq^\GG$.
The Euler class with constant $\Z$ coefficients can tell us only the parity
of the number of fixed points in each component.

\section{Comparison with Borel cohomology}\label{sec:Borel}

Borel cohomology was the first theory thought of as equivariant ordinary cohomology,
but is a considerably weaker theory than Bredon cohomology.
(See, for example, May's discussion in \cite{May:Borel}.)
There is a map from ordinary cohomology with $\Mackey\Z$ coefficients to Borel cohomology,
so the latter is also weaker than cohomology with $\Mackey\Z$ coefficients.
To see how much weaker, let us look at the calculation of $e(F)$ in Borel cohomology.

We take Borel cohomology to be the $RO(\GG$)-graded theory defined on based $\GG$-spaces by
\[
    \Borel_\GG^{RO(\GG)}(X) = H_\GG^{RO(\GG)}((E\GG)_+\smsh X),
\]
where, as usual, we use Burnside ring coefficients on the right, but suppress them from the notation.
(Because $E\GG$ is free, and $\Mackey A\to \Mackey \Z$ is an isomorphism at the $\GG/e$ level,
we could instead use $\Mackey\Z$ coefficients and get naturally isomorphic results.)
This is the usual Borel cohomology with $\Z$ coefficients but we have expanded the grading from
the common $\Z$ to $RO(\GG)$.
As shown in \cite{Co:BGU1preprint},
the Borel cohomology of a point is $\copt$ with $\xi$ inverted:
\[
    \Borel_\GG^{RO(\GG)}(S^0) \iso \Z[e, \xi, \xi^{-1}]/\langle 2e \rangle,
\]
where $\deg e = \sigma$ and $\deg \xi = 2\sigma - 2$ as before.
In the map $\copt \to \Borel_\GG^{RO(\GG)}(S^0)$, $\kappa$ goes to 0.
As with cohomology with $\Mackey\Z$ coefficients, Borel cohomology cannot tell the difference between $g$ and 2.

Note that, if we restrict to the $\Z$ grading, as is usually done, we get a polynomial
algebra in $e^2\xi^{-1}$ modulo $2e^2\xi^{-1} = 0$, a copy of the group cohomology of $\GG$ with $\Z$ coefficients.
If we restrict the grading to $\sigma+\Z$, we see the group cohomology of $\GG$ with twisted $\Z$ coefficients.
That the twisted and untwisted cohomologies can be combined in a single algebra like this seems to have
been first observed by \v{C}adek in \cite{Cad:BOnTwisted}.

Because the ordinary $\GG$-cohomology of $\Xpq pq$ is free over the cohomology of a point, we obtain its
Borel cohomology also by inverting $\xi$.
On doing so, the elements $\cxwt$ and $\cwt$ become invertible, with the result that, if we continued
to grade on $RO(\Pi B)$, the groups outside the $RO(\GG)$ grading would all be isomorphic to groups
in the $RO(\GG)$ grading via multiplication by an appropriate power of, say, $\cxwt$.
So we lose nothing by considering the $RO(\GG)$-graded part only.
To give the explicit result, let $\cd$ be the image of $\cxwt\cwd$ in $\Borel_\GG^{2\sigma}(\Xpq pq_+)$.
The following is then a corollary of Theorem~\ref{thm:cohomStructure}.

\begin{corollary}
Let $0 \leq p < \infty$ and $0 \leq q < \infty$ with $p+q > 0$.
Then $\Borel_{\GG}^{RO(\GG}(\Xpq{p}{q}_+)$ is a free module over $\Borel_{\GG}^{RO(\GG)}(S^0)$.
As a (graded) commutative algebra over $\Borel_{\GG}^{RO(\GG)}(S^0)$, 
$\Borel_{\GG}^{RO(\Pi B)}(\Xpq{p}{q}_+)$ is generated by $\cd$ in degree $2\sigma$,
which satisfies the single relation
\[
    \cd[p](\cd+e^2)^q = 0. 
\]
\qed
\end{corollary}

Of course, we could also use as a generator the element $c' = \xi^{-1}\cd$ in degree 2, but the
relation is then
\[
    (c')^p(c' + e^2\xi^{-1})^q = 0.
\]
For the simplicity of the relation, and to keep the generator more closely related to an element from
ordinary cohomology, we prefer to use $\cd$.

We view $\cd$ as the Euler class of $\omega\dual$. The Euler class of $\chi\omega\dual$ is then
$\cd + e^2$, the image of $\cwt\cxwd$. In doing this, we are choosing to say
that $\omega$ is a rank $2\sigma$ bundle over $E\GG\times \Xpq pq$.
Because $E\GG$ is free, we are as free to say $\omega$ has rank $2\sigma$ as to say it has rank 2.

Another way of seeing that $e(\chi\omega) = \cd + e^2$ is to recall that $\chi\omega = \omega\tensor_\C \C^{\sigma}$,
then use the additive formal group law of nonequivariant ordinary cohomology and the fact that $e(\C^\sigma) = e^2$.

Now consider the Euler classes of the bundles $O(d)$ and $\chi O(d)$, all of which we will think
of as having rank $2\sigma$. As a corollary of \cite[6.5]{CHTFiniteProjSpace},
or as a consequence of the formal group law for nonequivariant cohomology, we have the following.

\begin{proposition}
In the Borel cohomology of $\Xpq pq$, we have
\begin{align*}
    e(O(d)) &= d\cd \\ \intertext{and}
    e(\chi O(d)) &= d\cd + e^2
\end{align*}
for every $d\in\Z$.
\qed
\end{proposition}

\begin{theorem}[B\'ezout's Theorem for Borel Cohomology]
Let $F$ be as in the B\'ezout context~\ref{context}.
The Euler class of $F$ in the Borel cohomology of $\Xpq pq$ is
\[
    e_{BH}(F) = 
    \begin{cases}
        \Delta\cd[n] & \text{if $\Delta$, $\Delta_0$, and $\Delta_1$} \\
            & \text{\quad are even} \\
        \Delta\cd[n] + e^{2(n-n_0-n_1)}\cd[n_0](\cd + e^2)^{n_1}
            & \text{if $\Delta$ is even and} \\
            & \text{\quad $\Delta_0$ or $\Delta_1$ is odd } \\
        \Delta\cd[n_0](\cd + e^2)^{n-n_0} & \text{if $\Delta$ is odd.} \\
    \end{cases}
\]
\end{theorem}

\begin{proof}


These formulas can be derived from the preceding proposition or from
Theorem~\ref{thm:Zalgebraic}, using the fact that $\tau(1) = 2$
in Borel cohomology.
\end{proof}

As we saw with ordinary cohomology with $\Mackey\Z$ coefficients,
the Euler class in Borel cohomology contains significantly less information
than the one in ordinary cohomology with Burnside ring coefficients. 
The fixed point map would be
\begin{multline*}
   (-)^\GG\colon  H_\GG^{RO(\GG)}((E\GG)_+\smsh \Xpq pq_+) \\ \to H^\Z(((E\GG)_+\smsh \Xpq pq_+)^\GG;\Z)
    = H^\ZZ(*;\Z) = 0.
\end{multline*}
Thus, Borel cohomology contains no information at all about fixed points.





\bibliography{Bibliography}{}
\bibliographystyle{amsplain} 

\end{document}